\documentclass[a4paper]{amsart}

\usepackage[english]{babel} 

\usepackage{aeguill} 

\usepackage{tikz}
\usepackage{tikz-cd}
\usepackage{amsmath}
\usepackage{amssymb}
\usepackage{mathrsfs}
\usepackage{bm}
\usepackage{mathtools}
\usepackage{pigpen}
\usepackage{picture}
\usepackage{hyperref}
\usepackage{url}
\usepackage{graphicx}
\usepackage{csquotes}
\usepackage{cmll}

\numberwithin{equation}{section}
\newcommand{\normono}{\mathbin{\tikz[baseline] \draw[{Triangle[open,reversed,length=1.5mm,width=1.5mm]}->] (0pt,0.75ex) -- (2.5ex,0.75ex);}}
\newcommand{\norepi}{\mathbin{\tikz[baseline] \draw[-{Triangle[open,length=1.5mm,width=1.5mm]}] (0pt,0.75ex) -- (2.5ex,0.75ex);}}
\newcommand{\pullback}{\text{\pigpenfont A}}
\newcommand{\pushout}{\text{\pigpenfont I}}
\newcommand{\union}{\rotatebox{90}{$\triangleleft$}}
\newcommand{\inter}{\rotatebox{90}{$\triangleright$}}

\newcommand{\nsub}{\mathsf{NSub}}
\newcommand{\mon}{\mathsf{Mon}}
\newcommand{\cmon}{\mathsf{CMon}}
\newcommand{\ses}{\mathsf{SES}}
\newcommand{\nmono}{\mathsf{NMono}}
\newcommand{\nepi}{\mathsf{NEpi}}
\newcommand{\Ker}{\mathrm{ker}}
\newcommand{\Coker}{\mathrm{coker}}
\newcommand{\Def}[1]{\textit{\textbf{#1}}}
\newcommand{\id}{\mathrm{id}}
\newcommand{\map}{-{Straight Barb[length=2mm,width=2mm]}}
\newcommand{\mono}{{Straight Barb[reversed,length=2mm,width=2mm]}-{Straight Barb[length=2mm,width=2mm]}}

\newcommand{\normalmono}{{Triangle[open,reversed,length=2mm,width=2mm]}-{Straight Barb[length=2mm,width=2mm]}}
\newcommand{\normalepi}{-{Triangle[open,length=2mm,width=2mm]}}
\newcommand{\normaldemimono}{{Triangle[open,reversed,length=2mm,width=2mm]}-}
\newcommand{\X}{\mathscr{X}}
\newcommand{\R}{\mathcal{R}}
\newcommand{\noproof}{\hfill$\square$}

\theoremstyle{plain}
\newtheorem{proposition}{Proposition}[subsection]
\newtheorem{corollary}[proposition]{Corollary}
\newtheorem{theorem}[proposition]{Theorem}
\newtheorem{lemma}[proposition]{Lemma}

\theoremstyle{definition}
\newtheorem{definition}[proposition]{Definition}

\theoremstyle{remark}
\newtheorem{remark}[proposition]{Remark}
\newtheorem{example}[proposition]{Example}

\title{Di-Exact Categories and Lattices of Normal Subobjects}

\author{Florent Afsa}



\address{Institut de Recherche en Math\'ematique et Physique, Universit\'e catholique de Louvain, Chemin du Cyclotron 2, 1348 Louvain-la-Neuve, Belgium}

\email{florent.afsa@uclouvain.be}

\begin{document}
	
	\begin{abstract}
		The aim of this article is to study certain categorical-algebraic frameworks for basic homological algebra, introduced in \cite{PVdL3}, with the aim of better understanding the differences between them. We focus on \emph{homological self-duality}, \emph{preservation of normal maps by dinversion} and \emph{diexactness}, finding counterexamples that separate any two of these conditions. 	
		
		On the way, we encounter new examples and new characterizations, such as the \emph{Second Isomorphism Property}. We also consider homological self-duality in the context of regular categories. 
		
		Our main technique is to investigate the \emph{lattice of normal subobjects} of an object in any pointed category with kernels and cokernels. The category of \emph{monoidal semilattices} is a context where we have easy control  of these lattices. We show that properties of the homological frameworks under consideration may be expressed by means of lattice-theoretical properties, such as \emph{modularity} and \emph{distributivity}.
	\end{abstract}
	
	\maketitle
	
	\section{Introduction}
	
	\subsection{Context}
	
	In \cite{PVdL3}, the authors provide a foundation of homological algebra that extends the context of semi-abelian categories. They introduce new axiomatic frameworks that have three properties in common:
	\begin{enumerate}
		\item they are all weaker than the semi-abelian framework;
		\item they are all self-dual;
		\item they correspond to classical homological algebra lemmas.
	\end{enumerate}
	In this article, we are interested in four of those types of categories, which are defined in Section \ref{zexcat}.
	
	The first are \emph{z-exact categories}. They are pointed categories with kernels and cokernels. This is sufficient to have short exact sequences and a version of the Short $5$-Lemma.
	
	The second framework is \emph{homological self-duality}, which has many equivalent characterizations. One of them is the \emph{Pure Snake Lemma}, which is the Snake Lemma with an isomorphism in the middle. Another one is the Third Isomorphism Property.
	
	The third is when \emph{dinversion preserves normal maps} in the category. In this case, a version of the ($3\times 3$)-Lemma holds.
	
	The last, and maybe deepest one, is \emph{di-exactness}. In di-exact categories, the Normal Snake Lemma holds, and di-exactness is linked to semi-abelian categories \cite{Janelidze-Marki-Tholen} because a category with binary coproducts is semi-abelian if and only if it is di-exact and homological in the sense of Borceux and Bourn in \cite{Borceux-Bourn}.
	
	\subsection{In this article}
	
	Each one of these contexts is weaker than the next one, in the order we cited them. In \ref{sescmonhsd}, \ref{cmondpn} and \ref{exvect}, we give examples showing that these implications are strict.
	
	Moreover, we try in \ref{hsdreg} to situate them among other known conditions in the category theory literature, such as regular and normal categories.
	
	These considerations exhibit the efficient role of the lattice of normal subobjects. We show with Lemma \ref{uniinter} that this exists in every z-exact category. Then, in Theorems \ref{dexmod} and \ref{dexses}, we relate lattice theory properties to homological properties. In particular, we show in \ref{hsdses} and \ref{dexses} that homological self-duality and di-exactness in a category $\X$ are preserved when taking the category of short exact sequences of $\X$, once we assume that certain well-known conditions on the lattice of normal subobjects hold, namely modularity and distributivity.
	
	\section{z-Exact categories}\label{zexcat}
	
	In this section, we summarize some notions given in \cite{PVdL3}, needed for our purpose.
	Let $\X$ be a pointed category. We call $\X$ a \Def{z-exact category} \cite[Definition~1.5.1]{PVdL3} if it admits all kernels and all cokernels. Grandis calls it \emph{pointed semiexact} or \emph{p-semiexact} (\cite[1.3.3]{Grandis-HA2}).
	
	\subsection{Short exact sequences}
	
	\begin{definition}\cite[Definitions~1.4.3,~1.4.4]{PVdL3}
		A \Def{normal monomorphism} (respectively a \Def{normal epimorphism}) is a morphism which is a kernel of some morphism (respectively a cokernel of some morphism). We write $f\colon X\normono Y$ (respectively $f\colon X\norepi Y$) when $f$ is a normal monomorphism (respectively a normal epimorphism).
	\end{definition}
	
	\begin{definition}\cite[Definition~1.9.1]{PVdL3}
		A sequence of morphisms $X\overset{k}{\normono} Y\overset{q}{\norepi} Z$ is called a \Def{short exact sequence} if $k$ is a kernel of $q$ and $q$ is a cokernel of $k$.
	\end{definition}
	
	\begin{proposition}\cite[Proposition~1.4.7]{PVdL3}
		In a z-exact category, if $\kappa$ is a normal monomorphism, then $\Ker(\Coker(\kappa))=\kappa$; if $\pi$ is a normal epimorphism, then $\Coker(\Ker(\pi))=\pi$.\noproof
	\end{proposition}
	
	If $\X$ is a z-exact category, then the respective categories $\nmono(\X)$, $\nepi(\X)$ and $\ses(\X)$ of normal monomorphisms, normal epimorphisms and short exact sequences in $\X$ are equivalent \cite[1.10]{PVdL3}. This allows us to understand finite limits and colimits in these three categories by computing them in $\ses(\X)$. Here we focus on kernels and cokernels.
	
	\begin{proposition}\label{kercokerses}\cite[Corollary~1.10.3]{PVdL3}
		Suppose $\X$ is z-exact, and let $f\coloneq(\alpha,\beta,\gamma)$ be a morphism of short exact sequences.
		\[
			\begin{tikzpicture}
				\node(A) at (0,0){$\bullet$};
				\node(B) at (2,0){$\bullet$};
				\node(C) at (0,-2){$\bullet$};
				\node(D) at (2,-2){$\bullet$};
				\node(E) at (4,0){$\bullet$};
				\node(F) at (4,-2){$\bullet$};
				\draw[\normalmono] (A)--(B);
				\draw[\normalmono] (C)--(D);
				\draw[\map] (A)--(C) node[left,midway]{$\alpha$};
				\draw[\map] (B)--(D) node[left,midway]{$\beta$};
				\draw[\normalepi] (B)--(E);
				\draw[\normalepi] (D)--(F);
				\draw[\map] (E)--(F) node[left,midway]{$\gamma$};
			\end{tikzpicture}
		\]
		\begin{itemize}
			\item The kernel of $f$ exists in $\ses(\X)$ and is the short exact sequence $A\normono B\norepi C$ where $A$ is the kernel of $\alpha$, $B$ is the kernel of $\beta$, and $C$ is the cokernel of $A\normono B$.
			\item Dually, the cokernel of $f$ exists in $\ses(\X)$ and is the short exact sequence $A'\normono B'\norepi C'$ where $B'$ is the cokernel of $\beta$, $C'$ is the cokernel of $\gamma$, and $A'$ is the kernel of $B'\norepi C'$.
			\item $f$ is a normal monomorphism if and only if the left-hand square is a pullback of normal monomorphisms.
			\item Dually, $f$ is a normal epimorphism if and only if the right-hand square is a pushout of normal epimorphisms.\noproof
		\end{itemize}
	\end{proposition}
	
	Hence, $\ses(\X)$ is z-exact whenever $\X$ is. Given a normal monomorphism $X\normono Y$, we write $Y/X$ for its cokernel. In particular, we have a short exact sequence $X\normono Y\norepi Y/X$.
	
	\begin{definition}\label{totnorseq}\cite[Definition~2.5.1]{PVdL3}
		A \Def{totally normal sequence of monomorphisms} is a sequence of monomorphisms $X\normono Y\normono Z$ such that the composite $X\rightarrow Z$ is also a normal monomorphism.
	\end{definition}
	
	\subsection{Di-extensions}
	
	\begin{definition}\cite[Definition~1.8.1]{PVdL3}
		A map $f$ is \Def{normal} (or \Def{proper}) if it factors as $f=me$ where $m$ is a normal monomorphism and $e$ is a normal epimorphism. We call the pair $(m,e)$ a \Def{normal decomposition} of $f$. Whenever such a factorization exists, it is unique up to unique isomorphism. Then $m$ is a kernel of the cokernel of $f$ and $e$ is a cokernel of the kernel of $f$.
	\end{definition}
	
	\begin{definition}\cite[Definition~2.1.2]{PVdL3}
		A map $f$ is \Def{antinormal} if it factors as $f=em$ where e is a normal epimorphism and $m$ is a normal monomorphism. We call the pair $(e,m)$ an \Def{antinormal decomposition} of $f$. This time, such a factorization need not exist, neither be unique.
	\end{definition}
	
	\begin{definition}\cite[Definition~2.1.1]{PVdL3}
		A \Def{$(3\times 3)$-diagram} in $\X$ is a commutative diagram of the form
		\[
			\begin{tikzpicture}
				\node(A) at (0,0){$\bullet$};
				\node(B) at (2,0){$\bullet$};
				\node(C) at (4,0){$\bullet$};
				\node(D) at (0,-2){$\bullet$};
				\node(E) at (2,-2){$\bullet$};
				\node(F) at (4,-2){$\bullet$};
				\node(G) at (0,-4){$\bullet$};
				\node(H) at (2,-4){$\bullet$};
				\node(I) at (4,-4){$\bullet$};
				\draw[\map] (A)--(B);
				\draw[\map] (B)--(C);
				\draw[\map] (A)--(D);
				\draw[\map] (B)--(E);
				\draw[\map] (C)--(F);
				\draw[\map] (D)--(E);
				\draw[\map] (E)--(F);
				\draw[\map] (D)--(G);
				\draw[\map] (E)--(H);
				\draw[\map] (F)--(I);
				\draw[\map] (G)--(H);
				\draw[\map] (H)--(I);
			\end{tikzpicture}
		\]
		We call it a \Def{di-extension} if every row and every column is a short exact sequence.
	\end{definition}
	
	Now we explain how an antinormal map can yield a di-extension (\cite[Lemmas~2.1.3,~2.1.5]{PVdL3}). Suppose we have an antinormal map $\alpha=em$, and take $\mu$ a kernel of $e$ and $\varepsilon$ a cokernel of $m$. This yields an other antinormal map $\beta=\varepsilon\mu$, called a \Def{dinverse} of $\alpha$ (for \emph{di-inverse}). We construct the following diagram by taking the pullback of $m$ and $\mu$ in the top left and the pushout of $e$ and $\varepsilon$ in the bottom right. These exist via Lemma \ref{pullbacknoyau} below.
	
	\begin{equation}\label{diex}
		\begin{tikzpicture}[baseline=(N)]
			\node(A) at (0,0){$\bullet$};
			\node(B) at (2,0){$\bullet$};
			\node(C) at (4,0){};
			\node(D) at (0,-2){$\bullet$};
			\node(E) at (2,-2){$\bullet$};
			\node(F) at (4,-2){$\bullet$};
			\node(G) at (0,-4){};
			\node(H) at (2,-4){$\bullet$};
			\node(I) at (4,-4){$\bullet$};
			\node(L) at (0.5,-0.5){\pullback};
			\node(M) at (3.5,-3.5){\pushout};
			\node(N) at (-1,-2.125){};
			\draw[\normalmono] (A)--(B) node[above,midway]{$q$};
			\draw[\normalmono] (A)--(D) node[left,midway]{$p$};
			\draw[\normalmono] (B)--(E) node[right,midway]{$\mu$};
			\draw[\normalmono] (D)--(E) node[below,midway]{$m$};
			\draw[\normalepi] (E)--(F) node[above,midway]{$\varepsilon$};
			\draw[\normalepi] (E)--(H) node[left,midway]{$e$};
			\draw[\normalepi] (F)--(I) node[right,midway]{$\rho$};
			\draw[\normalepi] (H)--(I) node[below,midway]{$\pi$};
			\draw[\map,dashed] (B)--(F) node[above right,midway]{$\beta$};
			\draw[\map,dashed] (D)--(H) node[below left,midway]{$\alpha$};
		\end{tikzpicture}
	\end{equation}
	We can indeed show with Lemma \ref{pullbacknoyau} that $p=\Ker(\alpha)$, $q=\Ker(\beta)$, $\pi=\Coker(\alpha)$ and $\rho=\Coker(\beta)$. Thus, whenever $\alpha$ is a normal map, a normal factorization would be a cokernel of $p$ followed by a kernel of $\pi$. Similarly, whenever $\beta$ is a normal map, a normal factorization would be a cokernel of $q$ followed by a kernel of $\rho$. Hence, if $\alpha$ and $\beta$ are both normal, then we obtain a di-extension by adding normal factorizations of $\alpha$ and $\beta$ respectively in the bottom left corner and in the top right corner. The following definitions are contexts in which an antinormal map can generate a di-extension.
	
	\begin{definition}\cite[Proposition~2.5.7.(I)]{PVdL3}
		A z-exact category $\X$ is called \Def{homologically self-dual} when in the context of a diagram such as \eqref{diex}, if $\alpha$ is $0$, then $\beta$ is normal. We can formulate this as ``every antinormal decomposition of $0$ generates a di-extension''. We will sometime use the acronym HSD.
	\end{definition}
	
	Homological self-duality has many characterizations, such as the Pure Snake Lemma \cite[Proposition~2.5.7.(II)]{PVdL3}, but one that we are most interested in is the \textit{Third Isomorphism Property} \cite[Proposition~2.5.7.(III)]{PVdL3}: for every totally normal sequence of monomorphisms $X\normono Y\normono Z$ (Definition \ref{totnorseq}), the sequence $Y/X\rightarrow Z/X\rightarrow Z/Y$ is short exact. Since $Z/X\rightarrow Z/Y$ is always a normal epimorphism, this can be reformulated by saying that $Y/X\rightarrow Z/X$ is a normal monomorphism, or that $Y/X$ is a kernel of $Z/X\rightarrow Z/Y$.
	
	\begin{definition}\cite[Proposition~2.6.3.(I)]{PVdL3}
		A z-exact category $\X$ is called \Def{DPN} (for \Def{dinversion preserves normal maps}) whenever in the context of a diagram such as \eqref{diex}, the morphism $\alpha$ is normal if and only if $\beta$ is normal. We can formulate this as ``every antinormal map that is also normal generates a di-extension''.
	\end{definition}
	
	The preservation of normal maps by dinversion is the border case of the ($3\times3$)-lemma: in a ($3\times3$)-diagram, if three rows and two columns including the middle one are short exact sequences, then the third column is a short exact sequence \cite[Proposition~2.5.7.(II)-(V)]{PVdL3}.
	
	\begin{definition}\cite[Proposition~2.8.2.(II)]{PVdL3}
		A z-exact category $\X$ is called \Def{di-exact} if every antinormal map is normal. We can formulate this as ``every antinormal pair generates a di-extension''.
	\end{definition}
	
	It is clear that any di-exact category is DPN, and any DPN category is HSD. Besides, every semi-abelian category is di-exact. In fact, every homological category with binary coproducts is semi-abelian if and only if it is di-exact \cite[Theorem~5.7.3]{PVdL3}.
	
	\subsection{Some well-known lemmas}
	
	The following lemmas are well known; similar statements can be found, for instance, in~\cite{Borceux-Bourn}. In~\cite{PVdL3}, an explicit proof is given for each of these results.
	
	\begin{lemma}\label{normonocomp}\cite[Propositions~1.5.8.(iv),~1.5.9.(iii)]{PVdL3}
		Let $u$, $v$ be two composable morphisms. If $vu$ is a normal monomorphism and $v$ is  monic, then $u$ is a normal monomorphism. If $vu$ is a normal epimorphism and $u$ is  epic, then $v$ is a normal epimorphism.\noproof
	\end{lemma}
	
	\begin{lemma}\label{pullbacknormono}\cite[Proposition~1.5.4.(i),~1.7.1]{PVdL3}
		Any pullback of a normal monomorphism is a normal monomorphism. If $\X$ admits cokernels, any pullback of two normal monomorphisms yields a commutative square where the diagonal composite is a normal monomorphism.\noproof
	\end{lemma}
	
	\begin{lemma}\label{pullbacknoyau}\cite[Proposition~1.5.8.(i),~1.5.9.(i)]{PVdL3}
		Given two composable morphisms $f$ and $g$, the kernel of $gf$ is the pullback of the kernel of $g$ along $f$. Dually, the cokernel of $gf$ is the pushout of the cokernel of $f$ along $g$.\noproof
	\end{lemma}
	
	\begin{lemma}\label{pullbackiso}\cite[Proposition~1.6.1]{PVdL3}
		For a commutative diagram
		\[
			\begin{tikzpicture}
				\node(A) at (0,0){$\bullet$};
				\node(B) at (2,0){$\bullet$};
				\node(C) at (0,-2){$\bullet$};
				\node(D) at (2,-2){$\bullet$};
				\node(E) at (4,0){$\bullet$};
				\node(F) at (4,-2){$\bullet$};
				\draw[\normalmono] (A)--(B) node[above,midway]{$\mathrm{ker}(f)$};
				\draw[\normalmono] (C)--(D) node[below,midway]{$\mathrm{ker}(g)$};
				\draw[\map] (A)--(C) node[left,midway]{$\gamma$};
				\draw[\map] (B)--(D) node[left,midway]{$u$};
				\draw[\map] (B)--(E) node[above,midway]{$f$};
				\draw[\map] (D)--(F) node[below,midway]{$g$};
				\draw[\map] (E)--(F) node[left,midway]{$v$};
			\end{tikzpicture}
		\]
		if the right-hand square is a pullback, then $\gamma$ is an isomorphism.\noproof
	\end{lemma}
	
	\begin{lemma}\label{pullbackreco}\cite[Proposition~1.6.2]{PVdL3}
		Let $q\colon Y\rightarrow Z$, $q\colon Y'\rightarrow Z'$ be two morphisms, $k=\Ker(q)\colon X\normono Y$ and $k'\colon X'\rightarrowtail Z'$ a monomorphism such that $q'k'=0$, as written in the following diagram
		\[
			\begin{tikzpicture}
				\node(A) at (0,0){$X$};
				\node(B) at (2,0){$Y$};
				\node(C) at (0,-2){$X'$};
				\node(D) at (2,-2){$Y'$};
				\node(E) at (4,0){$Z$};
				\node(F) at (4,-2){$Z'$};
				\draw[\normalmono] (A)--(B) node[above,midway]{$k$};
				\draw[\mono] (C)--(D) node[below,midway]{$k'$};
				\draw[\map] (A)--(C) node[left,midway]{$u$};
				\draw[\map] (B)--(D) node[left,midway]{$v$};
				\draw[\map] (B)--(E) node[above,midway]{$q$};
				\draw[\map] (D)--(F) node[below,midway]{$q'$};
				\draw[\map] (E)--(F) node[left,midway]{$p$};
			\end{tikzpicture}
		\]
		If $p$ is a monomorphism, then the left-hand square is a pullback.\noproof
	\end{lemma}
	
	\section{Some algebraic preliminaries}
	
	\subsection{Monoids}
	
	Here we present a few results about monoids that will be useful next. They are explicit characterizations of normal monomorphisms and normal epimorphisms in the respective categories $\mon$ and $\cmon$ of monoids and commutative monoids. Since a morphism in $\mon$ is a monomorphism precisely when it is an injective map, we can view a monomorphism as an inclusion of a submonoid. Furthermore, the kernel of a map $f\colon M\rightarrow N$ is the inverse image of the neutral element of $N$.
	
	\begin{proposition} \cite{Pin}
		Let $K$ be a submonoid of a monoid $M$. Then $K$ is normal if and only if for all $k\in K$ and $x$, $y$ $\in M$, $xky\in K\Longleftrightarrow xy\in K$.
	\end{proposition}
	\begin{proof} Suppose $K$ is the kernel of some morphism $f$.
		For $k\in K$ and $x$, $y$ $\in M$, we have
		$$
			f(xky)=f(x)f(k)f(y)=f(x)f(y)=f(xy)
		$$
		so the condition holds.
		Conversely, we define a morphism that has kernel $K$ by taking the projection from $M$ to $M/\R$ where $\R$ is the equivalence relation defined for $m$, $n$ $\in M$ by $m\R n$ if and only if
		$$
			\forall x\textrm{, }y\in M\quad xmy\in K\Longleftrightarrow xny\in K
		$$
		Given $x$, $y$, $m$, $n$, $m'$, $n'$ $\in M$ such that $m\R m'$ and $n\R n'$, we have
		$$
			xmny\in K\Longleftrightarrow xmn'y\in K\Longleftrightarrow xm'n'y\in K
		$$
		so composition in $M$ is compatible with $\R$, hence $M/\R$ is a monoid and the projection is a morphism.
		Finally, we notice that the equivalence class of $1$ is exactly $K$, thus $K$ is the kernel of the projection.
	\end{proof}
	
	\begin{remark}\label{normonoid}
		In the case of commutative monoids, the condition changes as follows:
		$$
			\forall k\in K\quad \forall x\in M\quad x+k\in K\Longleftrightarrow x\in K
		$$
		Notice that the right-to-left implication is just the property of stability of the operation $+$.
	\end{remark}
	
	\begin{proposition}\label{cokercmon}
		For a normal submonoid $K$ of a commutative monoid $M$, the cokernel of the inclusion is the quotient $M/\R$ where $\R$ is the equivalence relation defined for $m$, $n$ $\in M$ by $m\R n$ if and only if there exist $k$, $l$ $\in K$ such that $m+k=n+l$. We write this quotient $M/K$.
	\end{proposition}
	\begin{proof}
		Let us first show that $\R$ is an equivalence relation which the operation of $M$ is compatible with. Reflexivity and symmetry are easy. For $m$, $n$ and $p$ in $M$ and $k$, $k'$, $l$ and $l'$ in $K$ such that $m+k=n+l$ and $n+k'=p+l'$, we have $m+k+k'=p+l+l'$, so transitivity holds. Now, if there are $m'$ and $n'$ in $M$ with $m+k=m'+k'$ and $n+l=n'+l'$, then $m+n+k+l=m'+n'+k'+l'$, so $(m+n)\R(m'+n')$.
		Now we check that $\pi\colon M\rightarrow M/\R$ is the cokernel of the inclusion $\iota$. The formula $k=0+k$ shows that $\pi\iota=0$. Let $f\colon M\rightarrow N$ be a morphism with $f\iota=0$. We define $\tilde{f}\colon  M/\R\rightarrow N$ for all $m\in M$ by $\tilde{f}([m])=f(m)$. By the hypothesis on $f$, the map $\tilde{f}$ is well defined, and $f=\tilde{f}\pi$. Clearly it is a morphism, and the only one satisfying this condition.
	\end{proof}

	\subsection{Monoidal semilattices}
	
	We introduce the notion of a monoidal semilattice, which is just a semilattice with a minimum. It is also known as \textit{semilattice with $0$} in the literature. For a general introduction to lattice theory, we refer to \cite{Birkhoff,Gratzer}.
	
	\begin{definition}
		A \Def{(join)-semilattice} is the data of a poset $(L,\leqslant)$ such that the map $(a,b)\mapsto\sup(a,b)$ is everywhere defined. We write $a\vee b\coloneq \sup(a,b)$ for $a$, $b$ in $L$. A \Def{lattice} is a semilattice such that the map $(a,b)\mapsto\inf(a,b)$ is everywhere defined. We write $a\wedge b\coloneq \inf(a,b)$ for $a$, $b$ in $L$.
	\end{definition}
	
	In this text, we are mainly interested in those semilattices that have a minimum, because these provide a structure of commutative monoid where the minimum is the neutral element. Given a lattice $(L,\vee,\wedge)$ with a minimum $0$, we will sometimes consider the underlying monoid $(L,\vee,0)$. Morphisms between two lattices with a minimum $0$ considered in the category $\mathsf{Lat}$ of lattices preserve $\wedge$ and $\vee$, while they instead preserve $\vee$ and $0$ in $\cmon$. To avoid ambiguity, we introduce the category $\mathsf{SLMon}$; this is the full subcategory of $\cmon$ whose objects are what we call \Def{monoidal semilattices}, which are by definition semilattices with a minimum.
	
	\begin{proposition}
		Any finite monoidal semilattice is a lattice.\noproof
	\end{proposition}
	
	Now we compute normal monomorphisms and normal epimorphisms of monoidal semilattices, but still in the category $\cmon$.
	
	\begin{proposition}\label{latmono}
		Let $L$ be a monoidal semilattice. A submonoid $K$ of $L$ is normal if and only if we have the following condition:
		$$
			\forall x\in L\quad\forall k \in K\quad x\leqslant k\Longrightarrow x\in K
		$$
		In particular, if $L$ is finite, then $K$ is normal if and only if it is of the form $\shpos a\coloneq\{k\in L\mid k\leqslant a\}$, for some $a\in L$.
	\end{proposition}
	\begin{proof}
		We use the characterization of \ref{normonoid}. Suppose $K$ is a normal submonoid of~$L$, and take $k\in K$ and $x\in L$ with $x\leqslant k$. Thus $x\vee k=k\in K$, so $x\in K$. Conversely, let $x\in L$ and $k\in K$. If $y\coloneq k\vee x\in K$, then $x\in K$ because $x\leqslant y$, so $K$ is normal. When $K$ is finite, $a\coloneq\sup K\in K$, so $K=\shpos a$.
	\end{proof}
	
	\begin{proposition}\label{cokerlattice}
		Let $k$ be an element of a monoidal semilattice $L$. The cokernel of $\kappa\colon\shpos k\normono L$ in $\cmon$ can be obtained as the sublattice $\shneg k\coloneq\{l\in L\mid k\leqslant l\}$ of~$L$. The canonical projection $L\norepi\shneg k$ is the map $\pi\colon l\mapsto l\vee k$. Thus, a short exact sequence of finite monoidal semilattices is completely determined by a pair $(L,k)$ where $k\in L$.
	\end{proposition}
	\begin{proof}
		We have $k$ is the neutral element of $\shneg k$, so $\pi\kappa=0$. Let $f\colon L\rightarrow M$ a morphism of commutative monoids such that $f\kappa=0$. We define $\tilde{f}\colon \shneg k\rightarrow M$ for all $x\in \shneg k$ by $\tilde{f}(x)=f(x)$. Thus, for all $x\in L$, we see that
		$$
			f(x)=f(x)+f(k)=f(x\vee k)=\tilde{f}\pi(x)
		$$
		and $\tilde{f}$ is clearly a morphism, and the only one satisfying $f=\tilde{f}\pi$.
	\end{proof}
	
	\begin{example}
		Consider the following monoidal semilattice $L$
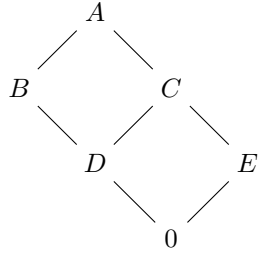
		\[
			\begin{tikzpicture}
				\node(A) at (0,0){$A$};
				\node(B) at (-1,-1){$B$};
				\node(C) at (1,-1){$C$};
				\node(D) at (0,-2){$D$};
				\node(E) at (2,-2){$E$};
				\node(F) at (1,-3){$0$};
				\draw[-,>=latex] (A)--(B);
				\draw[-,>=latex] (A)--(C);
				\draw[-,>=latex] (B)--(D);
				\draw[-,>=latex] (C)--(D);
				\draw[-,>=latex] (C)--(E);
				\draw[-,>=latex] (D)--(F);
				\draw[-,>=latex] (E)--(F);
			\end{tikzpicture}
		\]
		The normal submonoid given by $E$ is on the left below, and the quotient $\shneg E$ is on the right.
		
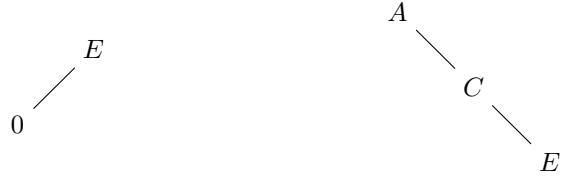
\[
			\begin{tikzpicture}
				\node(A) at (0,0){$E$};
				\node(B) at (-1,-1){$0$};
				\node(C) at (4,0.5){$A$};
				\node(D) at (5,-0.5){$C$};
				\node(E) at (6,-1.5){$E$};
				\draw[-,>=latex] (A)--(B);
				\draw[-,>=latex] (C)--(D);
				\draw[-,>=latex] (D)--(E);
			\end{tikzpicture}
		\]
		In the quotient, $A$ corresponds to the class of $A$ and $B$, $C$ corresponds to the class of $C$ and $D$, and $E$ corresponds to the class of $E$ and $0$.
	\end{example}
	
	\subsection{Lattices}
	
	We close this section recalling well-known characterizations of distributivity and modularity in a lattice. Remember that a lattice $(L,\vee,\wedge)$ is \Def{modular} whenever for every $x$, $y$ and $z$ in $L$ such that $x\geqslant z$, we have $x\wedge(y\vee z)=(x\wedge y)\vee z$, and is \Def{distributive} whenever for every $x$, $y$ and $z$ in $L$, the identity $x\wedge(y\vee z)=(x\wedge y)\vee (x\wedge z)$ holds. Any distributive lattice is modular.
	
	\begin{theorem}\label{thmlat}\cite[I,~Theorem~12,~I,~Theorem~13]{Birkhoff}
		A lattice is modular if and only if it does not contain the lattice on the left below (called \Def{pentagon}, and denoted $N(1,a,b,c,0)$) as a sublattice.
		
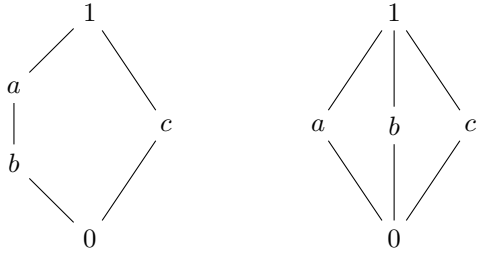
\[
			\begin{tikzpicture}
				\node(A) at (0,0){$1$};
				\node(B) at (-1,-1){$a$};
				\node(C) at (1,-1.5){$c$};
				\node(E) at (-1,-2){$b$};
				\node(H) at (0,-3){$0$};
				\draw[-,>=latex] (A)--(B);
				\draw[-,>=latex] (A)--(C);
				\draw[-,>=latex] (B)--(E);
				\draw[-,>=latex] (C)--(H);
				\draw[-,>=latex] (E)--(H);
				\node(A') at (4,0){$1$};
				\node(B') at (3,-1.5){$a$};
				\node(C') at (4,-1.5){$b$};
				\node(E') at (5,-1.5){$c$};
				\node(H') at (4,-3){$0$};
				\draw[-,>=latex] (A')--(B');
				\draw[-,>=latex] (A')--(C');
				\draw[-,>=latex] (A')--(E');
				\draw[-,>=latex] (B')--(H');
				\draw[-,>=latex] (C')--(H');
				\draw[-,>=latex] (E')--(H');
			\end{tikzpicture}
		\]
		A modular lattice is distributive if and only if it does not contain the lattice on the right above (called \Def{diamond}, and denoted $M(1,a,b,c,0)$) as a sublattice.\noproof
	\end{theorem}
	
	\begin{theorem}\label{cnsmod}\cite[I,~Theorem~13]{Birkhoff}
		A lattice $L$ is modular if and only if for every $x$, $y$ $\in L$, the map $\varphi\colon t\mapsto t\vee y$ is an isomorphism of lattices between $[x\wedge y,x]$ and $[y,x\vee y]$, where $[a,b]\coloneq\{k\in L\mid b\leqslant k\leqslant a\}$.
	\end{theorem}
	\begin{proof}
		Suppose first that $L$ is modular. For $t\in [x\wedge y,x]$, we have $(t\vee y)\wedge x=t\vee (y\wedge x)=t$, and for $u\in [y,x\vee y]$, we have $(u\wedge x)\vee y=u\wedge (x\vee y)=u$, so the map $\psi\colon u\mapsto u\wedge x$ is the inverse of $\varphi$. Since $\varphi$ commutes with $\vee$, the map $\psi$ does too, and likewise both $\varphi$ and $\psi$ commute with $\wedge$. Hence, they are isomorphisms of lattices.
		Now suppose that $L$ is not modular. Then $L$ contains a pentagon $N(1,a,b,c,0)$. Since $1=a\vee c=b\vee c$, the map $x\mapsto x\vee c$ from $[a\wedge c,a]$ to $[c,a\vee c]$ is not injective.
	\end{proof}
	
	\section{The lattice of normal subobjects}
	
	\subsection{Construction}\label{subsec4.1}
	
	We construct the lattice of normal subobjects of an object in any z-exact category. Let $\X$ be such a category, and $X$ an object of $\X$. A normal subobject of $X$ is an isomorphism class of normal monomorphisms with target $X$. For such a class represented by $y\colon Y\normono X$, we refer sometimes to $y$ or even to~$Y$. Denote $\nsub(X)$ the category whose objects are normal subobjects of $X$. A morphism between two normal subobjects is a morphism of $\X$ making the obvious triangle commute. For $Y\in\nsub(X)$, denote $\nsub(X\vert Y)$ the full subcategory of $\nsub(X)$ whose objects are those normal subobjects of $X$ which contain $Y$. The link between normal subobjects and di-extensions is that the data of an antinormal pair is essentially the same as the data of two normal subobjects of a given object, by taking for an antinormal pair $(e,m)$ the pair $(\Ker(e),m)$.
	
	\begin{proposition}
		In $\nsub(X)$, there is at most one morphism between any two objects, and it is determined by a normal monomorphism of $\X$.
	\end{proposition}
	\begin{proof}
		The first assertion is true by definition of a monomorphism and the second one comes from \ref{normonocomp}.
	\end{proof}
	
	\begin{lemma}\label{uniinter}
		Suppose $\X$ is z-exact. The product of the subobjects represented by any two morphisms $y\colon Y\normono X$ and $z\colon Z\normono X$ in $\nsub(X)$ exists and is represented by the pullback of $y$ and $z$ in $\X$. We write it $Y\inter Z$. The coproduct exists as well, and is represented by the kernel of the cokernel of the morphism $Y\rightarrow X/Z$. We write it $Y\union Z$.
	\end{lemma}
	\begin{proof}
		The pullback of $y$ and $z$ exists since it is the kernel of the composite $Y\normono X\norepi X/Z$ (Lemma \ref{pullbacknoyau}). We obtain a normal subobject because of \ref{pullbacknormono}. Let us prove the existence of the coproduct. Let $Q$ be the cokernel of $Y\rightarrow X/Z$. By Lemma \ref{pullbacknoyau}, it is the pushout of $X/Y$ and $X/Z$ under $X$, and it is also the cokernel of $Z\rightarrow X/Y$. Let $k\colon K\normono X$ be the kernel of $q\colon X\norepi Q$. Write also $\pi_Y$ and $\pi_Z$ for the respective cokernels of $y$ and $z$. The morphism $q$ factors through both $\pi_Y$ and $\pi_Z$, hence $qy=qz=0$. So $y$ and $z$ are both smaller than $k$.
		\[
			\begin{tikzpicture}
				\node(A) at (0,0){$Y$};
				\node(B) at (0,-4){$Z$};
				\node(D) at (0,-2){$K$};
				\node(E) at (4,-2){$X$};
				\node(F) at (6,-2){$Q$};
				\draw[\normalmono] (A)--(D);
				\draw[\normalmono] (B)--(D);
				\draw[\normalmono] (A)--(E) node[above,midway]{$y$};
				\draw[\normalmono] (B)--(E) node[below,midway]{$z$};
				\draw[\normalmono] (D)--(E) node[above,midway]{$k$};
				\draw[\normalepi] (E)--(F) node[above,midway]{$q$};
			\end{tikzpicture}
		\]
		To see that $k$ is the coproduct, take $t\colon T\normono X$  and write $p\colon X\norepi P$ its cokernel. Consider a cocone $Y\overset{y'}{\longrightarrow} T \overset{z'}{\longleftarrow} Z$ in $\nsub(X)$. Then $P$ is also the cokernel of $T\rightarrow X/Y$, hence $Y\rightarrow T\rightarrow P$ is $0$, so $p$ factors through $q$. This proves that $k$ factors through $t$.
	\end{proof}
	
	\begin{corollary}
		The category $\nsub(X)$ is finitely complete and finitely cocomplete.
	\end{corollary}
	
	\begin{proof}
		Lemma \ref{uniinter} shows that $\nsub(X)$ admits binary products and coproducts. The terminal object is $X$ while the initial object is $0$. Since any two parallel morphisms are equal, equalizers and coequalizers exist and are just identities.
	\end{proof}
	
	\begin{corollary}
		The category $\nsub(X)$ admits a lattice structure, with the product as the intersection and the coproduct as the union, and with $X$ as the maximum and $0$ as the minimum.\noproof
	\end{corollary}
	
	\begin{proposition}
		Given a normal subobject $x\colon X'\normono X$ of $X$, let $y\colon Y\normono X'$ and $z\colon Z\normono X'$ in $\nsub(X')$ be such that $xy$ and $xz$ are in $\nsub(X)$. Then the coproduct of $y$ and $z$ in $\nsub(X')$ has the same underlying object of $\X$ as the coproduct of $xy$ and $xz$ in $\nsub(X)$.
	\end{proposition}
	
	\begin{proof} 
		Let $k'\colon I'\rightarrow X'$ be the coproduct of $y$ and $z$ in $\nsub(X')$ and $k\colon I\rightarrow X$ the coproduct of $xy$ and $xz$ in $\nsub(X)$. Write $p'=\Coker(k')\colon X'\rightarrow Q'$, $p=\Coker(k)\colon X\rightarrow Q$. We obtain a commutative diagram as follows
		\[
			\begin{tikzpicture}
				\node(A) at (0,0){$I'$};
				\node(B) at (2,0){$X'$};
				\node(C) at (0,-2){$I$};
				\node(D) at (2,-2){$X$};
				\node(E) at (4,0){$Q'$};
				\node(F) at (4,-2){$Q$};
				\draw[\normalmono] (A)--(B) node[above,midway]{$k'$};
				\draw[\normalmono] (C)--(D) node[below,midway]{$k$};
				\draw[\map] (A)--(C) node[left,midway]{$f$};
				\draw[\normalmono] (B)--(D) node[left,midway]{$x$};
				\draw[\normalepi] (B)--(E) node[above,midway]{$p'$};
				\draw[\normalepi] (D)--(F) node[below,midway]{$p$};
				\draw[\map] (E)--(F) node[left,midway]{$q$};
			\end{tikzpicture}
		\]
		We want to show that $I$ is also a kernel of $p'$. Since $X'$ is a subobject of $X$ which contains $Y$ and $Z$, there is a normal monomorphism $\varphi\colon I\normono X'$ such that $x\varphi=k$. Moreover, $x\varphi f=kf=xk'$, so $\varphi f=k'$ since $x$ is monic. Now suppose there is some $g\colon T\rightarrow X'$ in $\X$ such that $p'g=0$; then $pxg=qp'g=0$, so there is some $\gamma\colon T\rightarrow I$ such that $xg=k\gamma$. Thus we have $x\varphi\gamma=xg$ and so $\varphi\gamma=g$. The uniqueness of $\gamma$ comes from the monic property of $\varphi$.
	\end{proof}

	\begin{proposition}\label{cokersquare}
		Given a commutative square of normal monomorphisms
		\[
			\begin{tikzpicture}
				\node(A) at (0,0){$W$};
				\node(B) at (2,0){$X$};
				\node(C) at (0,-2){$Y$};
				\node(D) at (2,-2){$Z$};
				\draw[\normalmono] (A)--(B) node[above,midway]{};
				\draw[\normalmono] (C)--(D) node[above,midway]{};
				\draw[\normalmono] (A)--(C) node[left,midway]{};
				\draw[\normalmono] (B)--(D) node[left,midway]{};
			\end{tikzpicture}
		\]
		suppose the diagonal composite is also normal. Then the cokernel of the induced map $Y/W\rightarrow Z/X$ is  $Z/(X\union Y)$.  By symmetry, it is also the cokernel of $X/W\rightarrow Z/Y$.
	\end{proposition}
	
	\begin{proof} 
		Since $Y\norepi Y/W$ is an epi, $Y/W\rightarrow Z/X$ has the same cokernel as $Y\rightarrow Z/X$, which is $Z/(X\union Y)$ by \ref{uniinter}.
	\end{proof}
	
	\begin{proposition}
		In the configuration of \ref{cokersquare}, if $\X$ is homologically self-dual, then the kernel of the map $Y/W\rightarrow Z/X$ is $(X/W)\inter(Y/W)\in\nsub(Z/W)$. By symmetry, it is also the kernel of $X/W\rightarrow Z/Y$.
	\end{proposition}
	\begin{proof}
		This is just \ref{pullbacknoyau} applied to the composite $Y/W\rightarrow Z/W\rightarrow Z/X$.
	\end{proof}
	
	\begin{remark}
		\label{nsubflattice}
		When a monoidal semilattice $L$ is finite, the map $l\mapsto L_l$ is an isomorphism of lattices from $L$ to $\nsub(L)$, with inverse $K\mapsto \max(K)$.
	\end{remark}
	
	\subsection{Examples with monoidal semilattices}
	
	The following example is the first one we constructed. In \cite{PVdL3}, the authors expressed their intuition that the category $\cmon$ should not be DPN, while they proved earlier that it is HSD. This is why we looked for a monoid that does not have cancellation (in the sense that $ab=ac$ does not imply $b=c$), and found what we wanted with a monoidal semilattice.
	
	\begin{example} \label{cmondpn}
		The category $\cmon$ of commutative monoids is homologically self-dual, but dinversion does not preserve normal maps in $\cmon$.
	\end{example}
	
	We explain here how we proceed, before giving the formal proof. As explained just above, we must start with two normal submonoids, say $X$ and $Y$, of a commutative monoid $Z$. Adding pullback, cokernels and pushout, we obtain a ($3\times 3$)-diagram
	\[
		\begin{tikzpicture}
			\node(A) at (0,0){$W$};
			\node(B) at (2,0){$X$};
			\node(C) at (4,0){$X/W$};
			\node(D) at (0,-2){$Y$};
			\node(E) at (2,-2){$Z$};
			\node(F) at (4,-2){$Z/Y$};
			\node(G) at (0,-4){$Y/W$};
			\node(H) at (2,-4){$Z/X$};
			\node(I) at (4,-4){$Z/X\union Y$};
			\node(J) at (0.5,-0.5){$\pullback$};
			\node(K) at (3.5,-3.5){$\pushout$};
			\draw[\normalmono] (A)--(B);
			\draw[\normalepi] (B)--(C);
			\draw[\normalmono] (A)--(D);
			\draw[\normalmono] (B)--(E);
			\draw[\map] (C)--(F);
			\draw[\normalmono] (D)--(E);
			\draw[\normalepi] (E)--(F);
			\draw[\normalepi] (D)--(G);
			\draw[\normalepi] (E)--(H);
			\draw[\normalepi] (F)--(I);
			\draw[\map] (G)--(H);
			\draw[\normalepi] (H)--(I);
		\end{tikzpicture}
	\]
	where $W=X\inter Y$. If we choose $Z$ to be a monoidal semilattice, we know with \ref{nsubflattice} that the top left square identifies a sublattice of $Z$, as follows
	\[
		\begin{tikzpicture}
			\node(A) at (0,0){$X\union Y$};
			\node(B) at (-1,-1){$X$};
			\node(C) at (1,-1){$Y$};
			\node(H) at (0,-2){$W$};
			\node(D) at (0,1){$Z$};
			\node(J) at (-0.3,-0.3){};
			\node(K) at (-0.5,-0.5){};
			\node(L) at (-0.7,-0.7){};
			\node(M) at (-0.7,-1.3){};
			\node(N) at (-0.5,-1.5){};
			\node(P) at (-0.3,-1.7){};
			\node(Q) at (0.3,-1.7){};
			\node(R) at (0.5,-1.5){};
			\node(S) at (0.7,-1.3){};
			\node(T) at (0.7,-0.7){};
			\node(U) at (0.5,-0.5){};
			\node(V) at (0.3,-0.3){};
			\draw[-,>=latex] (A)--(B);
			\draw[-,>=latex] (A)--(C);
			\draw[-,>=latex] (B)--(H);
			\draw[-,>=latex] (C)--(H);
			\draw[-,>=latex] (A)--(D);
			\draw[->,>=latex,dashed] (M)--(V);
			\draw[->,>=latex,dashed] (N)--(U);
			\draw[->,>=latex,dashed] (P)--(T);
			\draw[->,>=latex] (Q)--(L);
			\draw[->,>=latex] (R)--(K);
			\draw[->,>=latex] (S)--(J);
		\end{tikzpicture}
	\]
	By \ref{cokerlattice}, the map $Y/W\rightarrow Z/X$ is the ``projection'' of the interval $[X,W]$ to the interval $[X\union Y,Y]$, as displayed by the dashed arrows, while the map $X/W\rightarrow Z/Y$ is displayed by the plain arrows. Now the condition DPN demands that the first map is a normal monomorphism if and only if the second one is too. Since a normal monomorphism is in particular monic, and in our case an injection, we can adjust our monoidal semilattice $Z$ so that one ``projection'' is not injective while the other one is. See the proof below.
	\begin{proof}
		In \cite[Theorem~2.5.14]{PVdL3}, it is proved that $\cmon$ is homologically self-dual.
		Consider the pentagonal monoidal semilattice $L$
		\[
			\begin{tikzpicture}
				\node(A) at (0,0){$A$};
				\node(B) at (-1,-1){$B$};
				\node(C) at (1,-1.5){$D$};
				\node(E) at (-1,-2){$C$};
				\node(H) at (0,-3){$0$};
				\draw[-,>=latex] (A)--(B);
				\draw[-,>=latex] (A)--(C);
				\draw[-,>=latex] (B)--(E);
				\draw[-,>=latex] (C)--(H);
				\draw[-,>=latex] (E)--(H);
			\end{tikzpicture}
		\]
		With \ref{nsubflattice} we know that $L$ identifies to $\nsub(L)$. We have a pullback diagram
		
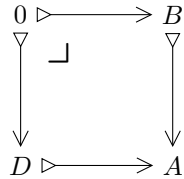
\[
			\begin{tikzpicture}
				\node(A) at (0,0){$0$};
				\node(B) at (2,0){$B$};
				\node(C) at (0,-2){$D$};
				\node(D) at (2,-2){$A$};
				\node(E) at (0.5,-0.5){$\pullback$};
				\draw[\normalmono] (A)--(B) node[above,midway]{};
				\draw[\normalmono] (C)--(D) node[above,midway]{};
				\draw[\normalmono] (A)--(C) node[left,midway]{};
				\draw[\normalmono] (B)--(D) node[left,midway]{};
			\end{tikzpicture}
		\]
		and we need to find the nature of the maps $D\rightarrow A/B$ and $B\rightarrow A/D$. The first one is an isomorphism, so it is in particular a normal monomorphism, and the second one sends both $B$ and $C$ to $A$, so it is not monic. This shows that the pullback square does not generate a di-extension.
	\end{proof}
	
	In the last proof, we saw that a diagram with normal monomorphisms is pretty much the same drawing as the representation of a lattice (up to a rotation). This will help us to find another required example: a z-exact category that is not HSD. In \cite[Exercise~2.8.9]{PVdL3}, the authors wonder whether $\ses(\X)$ is HSD whenever $\X$ is HSD.
	
	\begin{example}\label{sescmonhsd}
		The category $\ses(\cmon)$ is z-exact but not homologically self-dual.
	\end{example}
	
	Again, we explain how we found our example. Since $\ses(\cmon)$ is equivalent to $\nmono(\cmon)$, we take a totally normal sequence of normal monomorphisms as follows
	
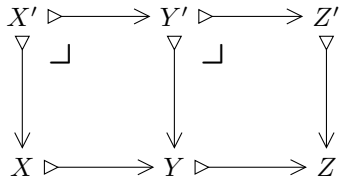
\[
		\begin{tikzpicture}
			\node(A) at (0,0){$X'$};
			\node(B) at (2,0){$Y'$};
			\node(C) at (4,0){$Z'$};
			\node(D) at (0,-2){$X$};
			\node(E) at (2,-2){$Y$};
			\node(F) at (4,-2){$Z$};
			\node(G) at (0.5,-0.5){$\pullback$};
			\node(H) at (2.5,-0.5){$\pullback$};
			\draw[\normalmono] (A)--(B);
			\draw[\normalmono] (B)--(C);
			\draw[\normalmono] (A)--(D);
			\draw[\normalmono] (B)--(E);
			\draw[\normalmono] (C)--(F);
			\draw[\normalmono] (D)--(E);
			\draw[\normalmono] (E)--(F);
		\end{tikzpicture}
	\]
	Both squares are pullback squares, and we can start with the following monoidal semilattice
	
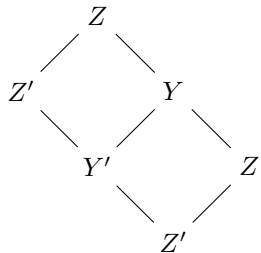
\[
		\begin{tikzpicture}
			\node(A) at (0,0){$Z$};
			\node(B) at (-1,-1){$Z'$};
			\node(C) at (1,-1){$Y$};
			\node(D) at (0,-2){$Y'$};
			\node(E) at (2,-2){$Z$};
			\node(F) at (1,-3){$Z'$};
			\draw[-,>=latex] (A)--(B);
			\draw[-,>=latex] (A)--(C);
			\draw[-,>=latex] (B)--(D);
			\draw[-,>=latex] (C)--(D);
			\draw[-,>=latex] (C)--(E);
			\draw[-,>=latex] (D)--(F);
			\draw[-,>=latex] (E)--(F);
		\end{tikzpicture}
	\]
	It remains to adjust it for our purpose, see the proof below.
	\begin{proof}
		Consider the pentagonal monoidal semilattice $L\coloneq N_5$
		\[
			\begin{tikzpicture}
				\node(A) at (0,0){$A$};
				\node(B) at (-1,-1){$B$};
				\node(C) at (1,-1.5){$D$};
				\node(E) at (-1,-2){$C$};
				\node(H) at (0,-3){$0$};
				\draw[-,>=latex] (A)--(B);
				\draw[-,>=latex] (A)--(C);
				\draw[-,>=latex] (B)--(E);
				\draw[-,>=latex] (C)--(H);
				\draw[-,>=latex] (E)--(H);
			\end{tikzpicture}
		\]
		With \ref{nsubflattice} we know that $L$ identifies to $\nsub(L)$. We have a totally normal sequence of monomorphisms in $\ses(\cmon)$
		\[(C,0)\normono(B,0)\normono(A,D) \]
		This notation comes from \ref{cokerlattice}.
		We need to understand the map $\gamma\colon (B,0)/(C,0)\rightarrow (A,D)/(C,0)$ and check whether it is a normal monomorphisms of short exact sequences.
		
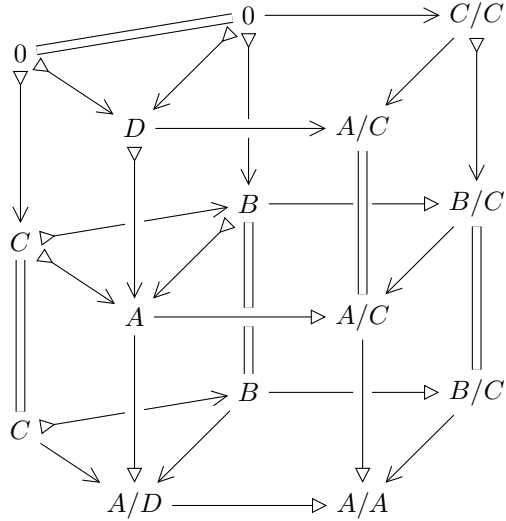
\[
			\begin{tikzpicture}
				\node(A) at (-3,-0.5){$0$};
				\node(B) at (0,0){$0$};
				\node(C) at (-1.5,-1.5){$D$};
				\node(D) at (-3,-3){$C$};
				\node(E) at (0,-2.5){$B$};
				\node(F) at (-1.5,-4){$A$};
				\node(G) at (-3,-5.5){$C$};
				\node(H) at (0,-5){$B$};
				\node(I) at (-1.5,-6.5){$A/D$};
				\node(J) at (3,0){$C/C$};
				\node(K) at (1.5,-1.5){$A/C$};
				\node(L) at (3,-2.5){$B/C$};
				\node(M) at (1.5,-4){$A/C$};
				\node(N) at (3,-5){$B/C$};
				\node(O) at (1.5,-6.5){$A/A$};
				\node(P) at (0,-1.5){};
				\node(Q) at (-1.5,-2.75){};
				\node(R) at (1.5,-2.5){};
				\node(S) at (0,-4){};
				\node(T) at (-1.5,-5.25){};
				\node(U) at (1.5,-5){};
				\draw[double,double distance=1mm] (A)--(B);
				\draw[\normalmono] (A)--(C);
				\draw[\normalmono] (B)--(C);
				\draw[\normalmono] (A)--(D);
				\draw[\normaldemimono] (B)--(P);
				\draw[\map] (P)--(E);
				\draw[\normalmono] (C)--(F);
				\draw[\normaldemimono] (D)--(Q);
				\draw[\map] (Q)--(E);
				\draw[\normalmono] (D)--(F);
				\draw[\normalmono] (E)--(F);
				\draw[double,double distance=1mm] (D)--(G);
				\draw[double,double distance=1mm] (E)--(S);
				\draw[double,double distance=1mm] (S)--(H);
				\draw[\normalepi] (F)--(I);
				\draw[\normaldemimono] (G)--(T);
				\draw[\map] (T)--(H);
				\draw[\map] (G)--(I);
				\draw[\map] (H)--(I);
				\draw[\map] (B)--(J);
				\draw[\map] (C)--(K);
				\draw[\map] (J)--(K);
				\draw[-] (E)--(R);
				\draw[\normalepi] (R)--(L);
				\draw[\normalepi] (F)--(M);
				\draw[\map] (L)--(M);
				\draw[-] (H)--(U);
				\draw[\normalepi] (U)--(N);
				\draw[\normalepi] (I)--(O);
				\draw[\map] (N)--(O);
				\draw[\normalmono] (J)--(L);
				\draw[double,double distance=1mm] (K)--(M);
				\draw[double,double distance=1mm] (L)--(N);
				\draw[\normalepi] (M)--(O);
			\end{tikzpicture}
		\]
		We have the three short exact sequences vertically on the left. To obtain the quotient $(A,D)/(C,0)$, by \ref{kercokerses} we take the cokernels of the respective maps $C\rightarrow A$ and $C\rightarrow A/D$, which are $A/C$ and $A/A$, and add the kernel of the induced map $A/C\rightarrow A/A$, thus we have the short exact sequence $(A/C,A/C)$. We do the same for $(B,0)/(C,0)$, and finally get our map $\gamma$ on the right face of the diagram.
		The top square on the right face is not a pullback, since the intersection of $A/C$ and $B/C$ is just $B/C$, so $\gamma$ is not a normal monomorphism by \ref{kercokerses}.
	\end{proof}
	
	\subsection{Di-exactness}
	
	The isomorphism given in the following lemma appears in what we call the \textit{Second Isomorphism Property}. Hence this theorem can be summarized by the categorical equation
	$$\textrm{di-exact}=\textrm{second iso property}+\textrm{third iso property}$$
	
	\begin{lemma}\label{cnsdiexact}
		Suppose $\X$ is homologically self-dual. $\X$ is di-exact if and only if for all $X\in\X$ and for all $Y$, $Z\in\nsub(X)$, we have $(Y\union Z)/Z\cong Y/(Y\inter Z)$.
	\end{lemma}
	\begin{proof}
		Take $X\in\X$ and $Y$, $Z\in\nsub(X)$. We find a $(3\times 3)$-diagram
		\[
			\begin{tikzpicture}
				\node(A) at (0,0){$Y\inter Z$};
				\node(B) at (2,0){$Y$};
				\node(C) at (4,0){$Y/(Y\inter Z)$};
				\node(D) at (0,-2){$Z$};
				\node(E) at (2,-2){$X$};
				\node(F) at (4,-2){$X/Z$};
				\node(G) at (0,-4){$Z/(Y\inter Z)$};
				\node(H) at (2,-4){$X/Y$};
				\node(I) at (4,-4){$X/(Y\union Z)$};
				\node(J) at (0.5,-0.5){$\pullback$};
				\node(K) at (3.5,-3.5){$\pushout$};
				\draw[\normalmono] (A)--(B);
				\draw[\normalepi] (B)--(C);
				\draw[\normalmono] (A)--(D);
				\draw[\normalmono] (B)--(E);
				\draw[\map] (C)--(F);
				\draw[\normalmono] (D)--(E);
				\draw[\normalepi] (E)--(F);
				\draw[\normalepi] (D)--(G);
				\draw[\normalepi] (E)--(H);
				\draw[\normalepi] (F)--(I);
				\draw[\map] (G)--(H);
				\draw[\normalepi] (H)--(I);
			\end{tikzpicture}
		\]
		The top left square is a pullback and the bottom right square is a pushout, so this is a di-extension if and only if the bottom and right sequences are short exact. But the kernel of $X/Z\norepi X/(Y\union Z)$ is $(Y\union Z)/Z$ by homological self-duality.
	\end{proof}
	
	Actually, we have a characterization of the Second Isomorphism Property in terms of which antinormal maps generate a di-extension:
	
	\renewcommand{\theenumi}{\roman{enumi}}
	
	\begin{proposition}\label{2isoprop}
		Let $\X$ be a z-exact category. Then the following are equivalent:
		\begin{enumerate}
			\item For every $X\in\X$, the Second Isomorphism Property holds in $\nsub(X)$;
			\item Every antinormal map in $\X$ with cokernel $0$ is normal;
			\item Every antinormal map in $\X$ with cokernel $0$ is a normal epimorphism.
		\end{enumerate}
	\end{proposition}
	\begin{proof}
		(i)$\implies$(iii) Let $f$ be an antinormal map in $\X$, with an antinormal decomposition $Y\normono X\norepi X/Z$, with $Y$ and $Z$ two normal subobjects of $X$. This yields a ($3\times 3$)-diagram
		\[
			\begin{tikzpicture}
				\node(A) at (0,0){$Y\inter Z$};
				\node(B) at (2,0){$Y$};
				\node(C) at (4,0){$Y/(Y\inter Z)$};
				\node(D) at (0,-2){$Z$};
				\node(E) at (2,-2){$X$};
				\node(F) at (4,-2){$X/Z$};
				\node(G) at (0,-4){$Z/(Y\inter Z)$};
				\node(H) at (2,-4){$X/Y$};
				\node(I) at (4,-4){$X/(Y\union Z)$};
				\node(J) at (0.5,-0.5){$\pullback$};
				\node(K) at (3.5,-3.5){$\pushout$};
				\draw[\normalmono] (A)--(B);
				\draw[\normalepi] (B)--(C)  node[above,midway]{$e$};
				\draw[\normalmono] (A)--(D);
				\draw[\normalmono] (B)--(E);
				\draw[\map] (C)--(F);
				\draw[\normalmono] (D)--(E);
				\draw[\normalepi] (E)--(F);
				\draw[\normalepi] (D)--(G);
				\draw[\normalepi] (E)--(H);
				\draw[\normalepi] (F)--(I)  node[right,midway]{$\Coker f$};
				\draw[\map] (G)--(H);
				\draw[\normalepi] (H)--(I);
				\draw[\map] (B)--(F)  node[above right,midway]{$f$};
			\end{tikzpicture}
		\]
		If $\Coker(f)=0$, it means that $X=Y\union Z$. Since $f$ is the composite $Y\norepi Y/(Y\inter Z)\rightarrow (Y\union Z)/Z$, $f$ is a normal epimorphism because $Y/(Y\inter Z)\rightarrow (Y\union Z)/Z$ is an isomorphism.
		
		(iii)$\implies$(ii) Is trivial.
		
		(ii)$\implies$(i) Let $X$ be an object of $\X$, and $Y$, $Z$ two normal subobjects of $X$. We have an antinormal map $f\colon Y\normono Y\union Z\norepi (Y\union Z)/Z$. Since $\Coker f=(Y\union Z)/(Y\union Z)=0$, the map $f$ is normal, and $\Ker(\Coker f)$ is an iso. Since $\Coker(\Ker f)=Y/(Y\inter Z)$, we obtain $Y/(Y\inter Z)\cong(Y\union Z)/Z$.
	\end{proof}
	
	We have a dual version of Proposition \ref{2isoprop}:
	
	\begin{proposition}
		Let $\X$ be a z-exact category. Then the following are equivalent:
		\begin{enumerate}
			\item For every $X\in\X$, the dual of the Second Isomorphism Property holds in $\nsub(X)$: for all $Y$, $Z\in\nsub(X)$ the respective kernels of $X/(Y\inter Z)\rightarrow X/Z$ and $X/Y\rightarrow X/(Y\union Z)$ are isomorphic;
			\item Every antinormal map in $\X$ with kernel $0$ is normal;
			\item Every antinormal map in $\X$ with kernel $0$ is a normal monomorphism.\noproof
		\end{enumerate}
	\end{proposition}
	
	\begin{theorem}\label{dexmod}
		If $\X$ is di-exact, then for every $X\in\X$, $\nsub(X)$ is modular.
	\end{theorem}
	\begin{proof}
		We have to combine \ref{cnsdiexact} and \ref{cnsmod}. For $T\in\nsub(X)$ such that $Y\inter Z\leqslant T\leqslant Y$, by homological self-duality, we have two short exact sequences
		\[
			T/(Y\inter Z)\normono Y/(Y\inter Z)\norepi Y/T
		\]
		and
		\[
			(T\union Z)/Z\normono (Y\union Z)/Z\norepi (Y\union Z)/(T\union Z)
		\]
		We already know that their middle terms are isomorphic. Moreover, since $T\inter Z=Y\inter Z$, we have $T/(Y\inter Z)=T/(T\inter Z)\cong(T\union Z)/Z$ so the left terms are also isomorphic. This proves that the right terms are isomorphic too. Since $(Y\union Z)/(T\union Z)=(Y\union T\union Z)/(T\union Z)\cong Y/(Y\inter(T\union Z))$, we finally obtain the equality $T=Y\inter(T\union Z)$. We prove similarly that $U=(Y\inter U)\union Z$ for $U\in \nsub(X)$ such that $Z\leqslant U\leqslant Y\union Z$ to conclude.
	\end{proof}
	
	\begin{theorem} \label{dexses}
		If $\X$ is di-exact, then $\ses(\X)$ is di-exact if and only if $\nsub(X)$ is distributive for every $X\in\X$.
	\end{theorem}
	\begin{proof}
		For a normal monomorphism $X\normono Y$ in $\X$, we write $XY$ for the short exact sequence $X\normono Y\norepi Y/X$. Suppose then that we have a pullback diagram in $\ses(\X)$
		\begin{equation}\label{pbses}
			\begin{tikzpicture}
				\node(A) at (0,0){$AB$};
				\node(B) at (2,0){$CD$};
				\node(C) at (0,-2){$EF$};
				\node(D) at (2,-2){$GH$};
				\node(E) at (0.5,-0.5){$\pullback$};
				\draw[\normalmono] (A)--(B) node[above,midway]{};
				\draw[\normalmono] (C)--(D) node[above,midway]{};
				\draw[\normalmono] (A)--(C) node[left,midway]{};
				\draw[\normalmono] (B)--(D) node[left,midway]{};
			\end{tikzpicture}
		\end{equation}
		Since every map is a normal monomorphism, by composition of pullbacks we know that $A$, $B$, $C$, $D$, $E$, $F$ and $G$ are normal subobjects of $H$, and that the identities $B=D\inter F$, $C=D\inter G$, $E=F\inter G$ and $A=B\inter C=B\inter E=C\inter E$ hold. We need to find a condition for $\Gamma\colon EF/AB\rightarrow GH/CD$ to be a normal monomorphism. This morphism is
		\[
			\begin{tikzpicture}
				\node(A) at (0,0){$(B\union E)/B$};
				\node(B) at (2.5,0){$F/B$};
				\node(C) at (5,0){$F/(B\union E)$};
				\node(D) at (0,-2){$(D\union G)/D$};
				\node(E) at (2.5,-2){$H/D$};
				\node(F) at (5,-2){$H/(D\union G)$};
				\draw[\normalmono] (A)--(B) node[above,midway]{};
				\draw[\normalepi] (B)--(C) node[above,midway]{};
				\draw[\map] (A)--(D) node[left,midway]{$\gamma$};
				\draw[\map] (B)--(E) node[left,midway]{$\gamma'$};
				\draw[\map] (C)--(F) node[left,midway]{$\gamma''$};
				\draw[\normalmono] (D)--(E) node[below,midway]{};
				\draw[\normalepi] (E)--(F) node[below,midway]{};
			\end{tikzpicture}
		\]
		Since $F/B=F/(D\inter F)\cong (D\union F)/D$, we know that $\gamma'$ is a normal monomorphism. By \ref{pullbacknormono}, if the left-hand square is a pullback, then the left-hand square is a pullback of normal monomorphisms in $\X$, and since $\X$ is di-exact, it means that $\gamma''$ is a normal monomorphism. Reciprocally, if $\gamma''$ is a normal monomorphism, Proposition \ref{pullbackreco} shows that the left-hand square is a pullback. Hence, $\Gamma$ is a normal monomorphism if and only if $\gamma''$ is a normal monomorphism, which is equivalent to $\gamma''=\Ker(\Coker(\gamma''))=\Ker(H/(D\union G)\norepi H/(D\union F\union G))$, so we want $F/(B\union E)\cong (D\union F\union G)/(D\union G)$. But $(D\union F\union G)/(D\union G)\cong F/(F\inter(D\union G))$ and $B\union E=(F\inter D)\union (F\inter G)$, proving that the original pullback in $\ses(\X)$ generates a di-extension if and only if $\nsub(H)$ is distributive.
	\end{proof}
	
	\begin{remark}\label{remdexses}
		Some partial reason for distributivity to appear is that the initial pullback diagram \eqref{pbses} consists of the data of the three normal subobjects $D$, $F$, $G$ of the object $H$; the other objects are just the intersection. We observe something similar in \ref{hsdses}.
	\end{remark}
	
	\begin{proposition}\label{dexsesdpn}
		If $\X$ is di-exact, then dinversion preserves normal maps in $\ses(\X)$.
	\end{proposition}
	\begin{proof}
		We use the notation of Theorem \ref{dexses}. Suppose that the map $\Gamma$ is a normal monomorphism. We want to show that $\Lambda\colon CD/AB\rightarrow GH/EF$ is also a normal monomorphism. $\Gamma$ represents the top rectangle of a di-extension, where the bottom left is both $(D\union F\union G)/(D\union F)$ and the cokernel of $\gamma$, \textit{i.e.}\ $(D\union G)/(D\union B\union E)=(D\union G)/(D\union(F\inter G))$. By \ref{cnsdiexact} and \ref{dexmod}, this is equal to
		$$
			(D\union(F\inter G)\union G)/(D\union(F\inter G))=G/(G\inter (D\union(F\inter G)))=G/((G\inter D)\union(F\inter G))
		$$
		Hence $D$ and $F$ have symmetric roles in this term, which means that it is equal to $(F\union G)/(F\union B\union C)$. Now consider the diagram below in which $\Lambda$ is the top rectangle
		\[
		\begin{tikzpicture}
				\node(A) at (0,0){$(B\union C)/B$};
				\node(B) at (4,0){$D/B$};
				\node(C) at (8,0){$D/(B\union C)$};
				\node(D) at (0,-2){$(F\union G)/F$};
				\node(E) at (4,-2){$H/F$};
				\node(F) at (8,-2){$H/(F\union G)$};
				\node(G) at (0,-4){$(D\union F\union G)/(D\union F)$};
				\node(H) at (4,-4){$H/(D\union F)$};
				\node(I) at (8,-4){$H/(D\union F\union G)$};
				\draw[\normalmono] (A)--(B) node[above,midway]{};
				\draw[\normalepi] (B)--(C) node[above,midway]{};
				\draw[\map] (A)--(D) node[left,midway]{$\lambda$};
				\draw[\map] (B)--(E) node[left,midway]{$\lambda'$};
				\draw[\map] (C)--(F) node[left,midway]{$\lambda''$};
				\draw[\normalmono] (D)--(E) node[below,midway]{};
				\draw[\normalepi] (E)--(F) node[below,midway]{};
				\draw[\map] (D)--(G) node[left,midway]{$p$};
				\draw[\normalepi] (E)--(H) node[left,midway]{};
				\draw[\normalepi] (F)--(I) node[left,midway]{};
				\draw[\normalmono] (G)--(H) node[below,midway]{};
				\draw[\normalepi] (H)--(I) node[below,midway]{};
			\end{tikzpicture}
		\]
		We have
		$$
			(B\union C)/B=C/(B\inter C)=C/(C\inter E)=(E\union C)/E
		$$
		and $(F\union G)/F=G/(F\inter G)=G/E$, so $\lambda$ is a normal monomorphism. Similarly, $\lambda'$ is also a normal monomorphism. Moreover, we just proved that $(D\union F\union G)/(D\union F)=(F\union G)/(F\union B\union C)$, which means that $p=\Coker(\lambda)$, so the left column is a short exact sequence. Hence, the top left-hand square is a pullback by \ref{pullbackreco}, so $\Lambda$ is a normal monomorphism.
	\end{proof}
	
	\begin{example}\label{exvect}
		Given a field $F$, the category $\mathsf{Vect}_F$ of $F$-vector spaces is di-exact, and the category $\ses(\mathsf{Vect}_F)$ is DPN but not di-exact.
	\end{example}
	\begin{proof}
		$\mathsf{Vect}_F$ is way more than di-exact, it is abelian. Hence, dinversion preserves normal maps in $\ses(\mathsf{Vect}_F)$ by \ref{dexsesdpn}. Furthermore, consider the $F$-vector space $F^2$; we have a diamond $M(F^2,G,H,K,0)$ (notation from \ref{thmlat}) in $\nsub(F^2)$ where $G$, $H$ and $K$ are three different one-dimensional subspaces of $F^2$. Hence $\nsub(F^2)$ is not distributive, and by \ref{dexses} the category $\ses(\mathsf{Vect}_F)$ is not di-exact. To see what happens in further details, consider the pullback of short exact sequences
		\[
			\begin{tikzpicture}
				\node(A) at (0,0){$(0,0)$};
				\node(B) at (2,0){$(H,0)$};
				\node(C) at (0,-2){$(K,0)$};
				\node(D) at (2,-2){$(F^2,G)$};
				\node(E) at (0.5,-0.5){$\pullback$};
				\draw[\normalmono] (A)--(B) node[above,midway]{};
				\draw[\normalmono] (C)--(D) node[above,midway]{};
				\draw[\normalmono] (A)--(C) node[left,midway]{};
				\draw[\normalmono] (B)--(D) node[left,midway]{};
			\end{tikzpicture}
		\]
		where $(V,W)$ denotes the short exact sequence $W\normono V\norepi V/W$. The map from $(K,0)/(0,0)$ to $(F^2,G)/(H,0)$ is
		\[
		\begin{tikzpicture}
			\node(A) at (0,0){$0$};
			\node(B) at (2.5,0){$K$};
			\node(C) at (5,0){$K$};
			\node(D) at (0,-2){$F^2/H$};
			\node(E) at (2.5,-2){$F^2/H$};
			\node(F) at (5,-2){$0$};
			\draw[\normalmono] (A)--(B) node[above,midway]{};
			\draw[double,double distance=1mm] (B)--(C) node[above,midway]{};
			\draw[\map] (A)--(D) node[left,midway]{};
			\draw[\map] (B)--(E) node[left,midway]{};
			\draw[\map] (C)--(F) node[left,midway]{};
			\draw[double,double distance=1mm] (D)--(E)	 node[below,midway]{};
			\draw[\normalepi] (E)--(F) node[below,midway]{};
		\end{tikzpicture}
		\]
		which is not a normal monomorphism since the left-hand square is not a pullback (\ref{kercokerses}).
	\end{proof}
	
	Let summarize what we know about di-exactness so far. Let $\X$ be a z-exact category. We write $\mathsf{DEx}(\X)$ for the category whose objects are di-extensions of $\X$ and whose morphisms are those $9$-tuples of morphisms of $\X$ which make everything commute. In general it is not true that $\mathsf{DEx}(\X)$ is equivalent to $\ses^2(\X)\coloneq\ses(\ses(\X))$; this is precisely the case when $\X$ is di-exact. Now define a tri-extension to be a ($3\times3\times3$)-diagram of which each line is a short exact sequence, and $\mathsf{TEx}(\X)$ to be the category of tri-extension in $\X$, and suppose that $\X$ is di-exact. It can be seen easily that $\mathsf{DEx}(\ses(\X))$ is equivalent to $\mathsf{TEx}(\X)$. However, it is not equivalent to $\ses(\mathsf{DEx}(\X))=\ses(\ses^2(\X))=\ses^3(\X)$, and this is again because we need precisely di-exactness of $\ses(\X)$, which happen by \ref{dexses} when the lattice of normal subobjects is distributive. Now we can wonder whether we need to add some conditions on the lattice of normal subobjects so that $\ses^n(\X)$ is equivalent to the category of $n$-fold extension on $\X$. The following proposition shows that distributivity in the lattice of normal subobjects remains when taking higher extensions.
	
	\begin{proposition}\label{dexsesn}
		Let $\X$ be a z-exact category. Then $\nsub(X)$ is distributive for every $X\in\X$ if and only if $\nsub(S)$ is distributive for every $S\in\ses(\X)$. When this happens, $\nsub(S)$ is distributive for every $S\in\ses^n(\X)$.
	\end{proposition}
	\begin{proof}
		We need to find an explicit formulation of the distributivity of the lattice of normal subobjects of an object in $\ses(\X)$. We use the notations of the proof of \ref{dexses}. Let $VW$ be a short exact sequence, and $X'X$, $Y'Y$ and $Z'Z$ be three normal subobjects of $VW$. All eight objects are normal subobjects of $W$. To first compute $X'X\union Y'Y$, we have that the cokernel $C$ of $X'X\rightarrow VW/Y'Y$ is
		$$
			C\coloneq(V\union X\union Y)/(X\union Y)\normono W/(X\union Y)\norepi W/(V\union X\union Y)
		$$
		Hence $X'X\union Y'Y$, which is the kernel of $VW\norepi C$, is
		$$
			V\inter(X\union Y)\normono X\union Y\norepi (X\union Y)/(V\inter(X\union Y))
		$$
		To find it, we take for the middle term the kernel of $W\norepi W/X\union Y$, and for the left term the pullback of $V$ and $X\union Y$ above $W$, because of \ref{kercokerses}. We obtain similarly $X'X\union Z'Z$, and thus the expression of $(X'X\union Y'Y)\inter (X'X\union Z'Z)$ is
		$$
			V\inter(X\union Y)\inter(X\union Z)\normono (X\union Y)\inter(X\union Z)\norepi Q \label{(xy)(xz)}
		$$
		where $Q$ is the cokernel of the morphism on the left (we don't need it explicitly).
		Now $Y'Y\inter Z'Z$ is
		$$
			V\inter Y\inter Z\normono Y\inter Z\norepi (Y\inter Z)/(V\inter Y\inter Z)
		$$
		so we have for $X'X\union(Y'Y\inter Z'Z)$
		$$
			V\inter(X\union(Y\inter Z))\normono X\union (Y\inter Z)\norepi Q' \label{x(yz)}
		$$
		Comparing $(X'X\union Y'Y)\inter (X'X\union Z'Z)$ and $X'X\union(Y'Y\inter Z'Z)$, we conclude that distributivity holds in $\nsub(W)$ if and only if it holds in $\nsub(VW)$.
	\end{proof}
	
	Therefore, by Theorem \ref{dexses}, it suffices that $\X$ and $\ses(\X)$ are di-exact for $\ses^n(\X)$ to be di-exact for every $n$, and so the $n$-fold extensions on $\X$ are the objects of $\ses^n(\X)$.
	
	\subsection{Homological self-duality in the regular case}\label{hsdreg}
	
	Here we explore the links between homological self-duality and regular categories \cite{Barr-Grillet-vanOsdol}, of which we recall briefly the definition.
	
	\begin{definition}
		In a category $\X$, a \Def{regular epimorphism} is a morphism which is a coequalizer of some pair of parallel morphisms.
	\end{definition}
	
	Any regular epimorphism is an epimorphism.
	
	\begin{definition}
		A category $\X$ is called \Def{regular} whenever:
		\begin{itemize}
			\item $\X$ has all finite limits;
			\item every kernel pair has a coequalizer;
			\item every regular epimorphism is stable under pullback.
		\end{itemize}
	\end{definition}
	
	The following lemma is a characterization of homological self-duality in the context of a regular z-exact category.
	
	\begin{lemma}\label{hsdstablenepi}
		Suppose $\X$ is pointed, regular, and admits cokernels. Then the following are equivalent:
		\begin{enumerate}
			\item $\X$ is homologically self-dual;
			\item every pullback in $\X$ of a normal epimorphism along a normal monomorphism is a normal epimorphism.
		\end{enumerate}
	\end{lemma}
	\begin{proof}
		(ii)$\implies$(i) $\X$ is z-exact since it has all finite limits and all cokernels. Let $X\normono Y\normono Z$ be a totally normal sequence of monomorphisms. We have a morphism of short exact sequences
		\[
			\begin{tikzpicture}
				\node(A) at (0,0){$Y$};
				\node(B) at (2,0){$Z$};
				\node(C) at (0,-2){$K$};
				\node(D) at (2,-2){$Z/X$};
				\node(E) at (4,0){$Y/Z$};
				\node(F) at (4,-2){$Y/Z$};
				\draw[\normalmono] (A)--(B);
				\draw[\normalmono] (C)--(D);
				\draw[\map] (A)--(C) node[left,midway]{$\gamma$};
				\draw[\normalepi] (B)--(D);
				\draw[\normalepi] (B)--(E);
				\draw[\normalepi] (D)--(F);
				\draw[double,double distance=1mm] (E)--(F);
			\end{tikzpicture}
		\]
		From \ref{pullbackreco} we deduce that the left-hand square is a pullback. Thus $\gamma$ is a normal epimorphism, and so $Y\rightarrow Z/X$ is normal.
		
		(i)$\implies$(ii) We start with a commutative diagram where the bottom is a short exact sequence, the right-hand rectangle is a pullback and $k$ is the kernel of $\rho$.
		\[
			\begin{tikzpicture}
				\node(A) at (-2,0){$K$};
				\node(B) at (2,0){$P$};
				\node(C) at (-2,-2){$X$};
				\node(D) at (2,-2){$Y$};
				\node(E) at (6,0){$T$};
				\node(F) at (6,-2){$Y/X$};
				\node(G) at (4,-1){$P/K$};
				\draw[\normalmono] (A)--(B) node[above,midway]{$k$};
				\draw[\normalmono] (C)--(D) node[below,midway]{$\iota$};
				\draw[\map] (A)--(C) node[left,midway]{$\gamma$};
				\draw[\normalmono] (B)--(D) node[left,midway]{$p$};
				\draw[\map] (B)--(E) node[above,midway]{$\rho$};
				\draw[\normalepi] (D)--(F) node[below,midway]{$\pi$};
				\draw[\normalmono] (E)--(F) node[right,midway]{$f$};
				\draw[\map,dashed] (G)--(E) node[below,midway]{$\bar{\rho}$};
				\draw[\normalepi,dashed] (B)--(G) node[below left,midway]{$q$};
				\draw[\map,dashed] (G)--(F) node[below left,midway]{$\alpha$};
			\end{tikzpicture}
		\]
		By \ref{pullbacknormono}, $p$ is a normal monomorphism, and by \ref{pullbackiso}, $\gamma$ is an isomorphism. Moreover, since $\pi$ is a regular epimorphism, $\rho$ is also a regular epimorphism, and $\rho$ factors as $\rho=\bar{\rho}q$ where $q$ is the cokernel of $k$. $\bar{\rho}$ is a regular epimorphism, and it remains for us to show that it is a monomorphism. By homological self-duality, $\alpha\colon P/K\rightarrow Y/X$  is a normal monomorphism, and we have
		$$
			\alpha q=\pi p=f\rho=f\bar{\rho}q
		$$
		So $\alpha=f\bar{\rho}$ because $q$ is epic, and thus $\bar{\rho}$ is a monomorphism.
	\end{proof}
	
	\begin{proposition}\label{propnsub}
		Let $X\normono Y$ be a normal monomorphism in a z-exact category $\X$. We define $\Phi\colon \nsub(Y/X)\rightarrow\nsub(Y\vert X)$ (see notations at the beginning of \ref{subsec4.1}) for every $t\colon T\normono Y/X$ by the pullback of $t$ along $Y\norepi Y/X$, and $\Psi\colon \nsub(Y\vert X)\rightarrow\nsub(Y/X)$ for every $u\colon U\normono Y$ that contains $X$ by the kernel of $Y/X\rightarrow Y/U$. Then
		\begin{enumerate}
			\item $\Phi$ is a right adjoint of $\Psi$ and $\Phi\Psi=\id$;
			\item $\Phi$ preserves intersection, $\Psi$ preserves union;
			\item if $\X$ is homologically self-dual, then $\Psi$ is the functor $U\mapsto U/X$ and we have the formula
			\[
				(U\union V)/X=(U/X)\union(V/X)
			\]
			for all $U$, $V$ in $\nsub(Y\vert X)$;\label{3hsd}
			\item if $\X$ is homologically self-dual and regular, then (\ref{3hsd}) is true, $\Phi$ and $\Psi$ are isomorphism of lattices, and we have the formula
			\[
				(U\inter V)/X=(U/X)\inter (V/X)
			\]
		\end{enumerate}
	\end{proposition}
	\begin{proof}
		\begin{enumerate}
			\item Let $U$ be in $\nsub(Y\vert X)$ and $T$ be in $\nsub(Y/X)$. Consider a morphism $f\colon U\rightarrow P(T)$. We have a commutative diagram
			\[
				\begin{tikzpicture}
					\node(A) at (0,0){$P(T)$};
					\node(B) at (2,0){$Y$};
					\node(C) at (0,-2){$T$};
					\node(D) at (2,-2){$Y/X$};
					\node(E) at (4,0){$Y/U$};
					\node(F) at (4,-2){$Q$};
					\node(G) at (-2,0){$U$};
					\node(H) at (0,-4){$R(U)$};
					\node(I) at (0.5,-0.5){$\pullback$};
					\draw[\normalmono] (A)--(B) node[above,midway]{};
					\draw[\normalmono] (C)--(D) node[below,midway]{$t$};
					\draw[\map] (A)--(C) node[left,midway]{};
					\draw[\normalepi] (B)--(D) node[left,midway]{};
					\draw[\normalepi] (B)--(E) node[above,midway]{};
					\draw[\normalmono] (G)--(A) node[above,midway]{$f$};
					\draw[\normalepi] (D)--(F) node[below,midway]{$q$};
					\draw[\normalmono] (H)--(D) node[below right,midway]{$m$};
					\draw[\normalepi] (D)--(E) node[left,midway]{};
					\draw[\map,dashed] (E)--(F) node[right,midway]{$\varphi$};
					\draw[\map,dashed] (H)--(C) node[left,midway]{};
				\end{tikzpicture}
			\]
			where $q$ is the cokernel of $t$. The universal property of $Y/U$ provides $\varphi$ because the map $U\rightarrow Q$ is $0$, so we observe that $qm=0$, and we obtain the adjunct of $f$.
			Now suppose there is a morphism $g\colon R(U)\rightarrow T$. This provides a morphism $U\rightarrow T$ such that $U\rightarrow T\rightarrow Y/X$ is equal to $U\rightarrow Y\rightarrow Y/X$, so by universal property we have the adjunct $U\rightarrow P(T)$.
			Finally the diagram below shows that $PR(U)=U$ because the two rows are short exact and the right map is an isomorphism, so the left square is a pullback (\ref{pullbackreco}).
			\[
				\begin{tikzpicture}
					\node(A) at (0,0){$U$};
					\node(B) at (2,0){$Y$};
					\node(C) at (0,-2){$R(U)$};
					\node(D) at (2,-2){$Y/X$};
					\node(E) at (4,0){$Y/U$};
					\node(F) at (4,-2){$Y/U$};
					\draw[\normalmono] (A)--(B) node[above,midway]{};
					\draw[\normalmono] (C)--(D) node[below,midway]{};
					\draw[\map] (A)--(C) node[left,midway]{};
					\draw[\normalepi] (B)--(D) node[left,midway]{};
					\draw[\normalepi] (B)--(E) node[above,midway]{};
					\draw[\normalepi] (D)--(F) node[below,midway]{};
					\draw[double,double distance=1mm] (E)--(F) node[left,midway]{};
				\end{tikzpicture}
			\]
			\item Is true because $\Phi$ is a right adjoint and intersection is product, and $\Psi$ is a left adjoint and union is coproduct.
			\item Holds by definition of homological self-duality, and the formula comes from the previous point.
			\item For $T$ in $\nsub(Y/X)$, the map $\Phi(T)\rightarrow T$ is a normal epimorphism because of \ref{hsdstablenepi}, and its kernel is $X$, so by homological self-duality, $T=\Psi(\Phi(T))$. Now $\Psi$ is an isomorphism of categories, so it preserves both products and coproducts, which are respectively intersection and union in the lattice structure of $\nsub(Y/X)$ and $\nsub(Y\vert X)$.\qedhere
		\end{enumerate}
	\end{proof}
	
	The following theorem is a kind of analog of Theorem \ref{dexses} where homological self-duality replaces di-exactness and modularity replaces distributivity. However, it is weaker because there is no equivalence any more and we need regularity. It helps, nonetheless, to understand why we need a lattice that is not modular in Example \ref{sescmonhsd}.
	
	\begin{theorem}\label{hsdses}
		Suppose $\X$ is homologically self-dual and regular. If $\nsub(X)$ is modular for every $X$ in $\X$, then $\ses(\X)$ is homologically self-dual.
	\end{theorem}
	\begin{proof}
		Consider a totally normal sequence of monomorphisms $X'X\normono Y'Y\normono Z'Z$ in $\ses(\X)$, which is equivalent to $\nmono(\X)$ (notations as in the proof of \ref{dexses}).
		\begin{equation}\label{nmonoseq}
			\begin{tikzpicture}[baseline=(I)]
				\node(A) at (0,0){$X'$};
				\node(B) at (2,0){$Y'$};
				\node(C) at (0,-2){$X$};
				\node(D) at (2,-2){$Y$};
				\node(E) at (4,0){$Z'$};
				\node(F) at (4,-2){$Z$};
				\node(G) at (0.5,-0.5){\pullback};
				\node(H) at (2.5,-0.5){\pullback};
				\node(I) at (0,-1.125){};
				\draw[\normalmono] (A)--(B) node[above,midway]{};
				\draw[\normalmono] (C)--(D) node[below,midway]{};
				\draw[\normalmono] (A)--(C) node[left,midway]{};
				\draw[\normalmono] (B)--(D) node[left,midway]{};
				\draw[\normalmono] (B)--(E) node[above,midway]{};
				\draw[\normalmono] (D)--(F) node[below,midway]{};
				\draw[\normalmono] (E)--(F) node[left,midway]{};
			\end{tikzpicture}
		\end{equation}
		To be able to understand the map $\Gamma\colon Y'Y/X'X\rightarrow Z'Z/X'X$, we study the diagram below where the three short exact sequences stand vertically on the left.
		\[
			\begin{tikzpicture}
				\node(A) at (-3,-0.5){$X'$};
				\node(B) at (0,0){$X$};
				\node(C) at (-1.5,-1.5){$Z'$};
				\node(D) at (-3,-3){$X$};
				\node(E) at (0,-2.5){$Y$};
				\node(F) at (-1.5,-4){$Z$};
				\node(G) at (-3,-5.5){$X/X'$};
				\node(H) at (0,-5){$Y/Y'$};
				\node(I) at (-1.5,-6.5){$Z/Z'$};
				\node(J) at (3,0){$(X\union Y')/X$};
				\node(K) at (1.5,-1.5){$(X\union Z')/X$};
				\node(L) at (3,-2.5){$Y/X$};
				\node(M) at (1.5,-4){$Z/X$};
				\node(N) at (3,-5){$Y/(X\union Y')$};
				\node(O) at (1.5,-6.5){$Z/(X\union Z')$};
				\node(P) at (0,-1.5){};
				\node(Q) at (-1.5,-2.75){};
				\node(R) at (1.5,-2.5){};
				\node(S) at (0,-4){};
				\node(T) at (-1.5,-5.25){};
				\node(U) at (1.5,-5){};
				\draw[\normalmono] (A)--(B);
				\draw[\normalmono] (A)--(C);
				\draw[\normalmono] (B)--(C);
				\draw[\normalmono] (A)--(D);
				\draw[\normaldemimono] (B)--(P);
				\draw[\map] (P)--(E);
				\draw[\normalmono] (C)--(F);
				\draw[\normaldemimono] (D)--(Q);
				\draw[\map] (Q)--(E);
				\draw[\normalmono] (D)--(F);
				\draw[\normalmono] (E)--(F);
				\draw[\normalepi] (D)--(G);
				\draw[-] (E)--(S);
				\draw[\normalepi] (S)--(H);
				\draw[\normalepi] (F)--(I);
				\draw[-] (G)--(T);
				\draw[\map] (T)--(H);
				\draw[\map] (G)--(I);
				\draw[\map] (H)--(I);
				\draw[\map] (B)--(J);
				\draw[\map] (C)--(K);
				\draw[\map] (J)--(K);
				\draw[-] (E)--(R);
				\draw[\normalepi] (R)--(L);
				\draw[\normalepi] (F)--(M);
				\draw[\normalmono] (L)--(M);
				\draw[-] (H)--(U);
				\draw[\normalepi] (U)--(N);
				\draw[\normalepi] (I)--(O);
				\draw[\map] (N)--(O);
				\draw[\normalmono] (J)--(L);
				\draw[\normalmono] (K)--(M);
				\draw[\normalepi] (L)--(N);
				\draw[\normalepi] (M)--(O);
			\end{tikzpicture}
		\]
		Here $\Gamma$ is on the right-hand side of the diagram. We have a normal monomorphism $Y/X\normono Z/X$ because $\X$ is homologically self-dual, and so we need the top square on the right-hand side to be a pullback. Since $Y'=Z'\inter Y$ and by modularity we have
		$$
			(X\union Y')/X=(X\union(Z'\inter Y))/X=((X\union Z')\inter Y)/X
		$$
		which is equal to $((X\union Z')/X)\inter (Y/X)$ because of \ref{propnsub}.
	\end{proof}
	
	\begin{remark}
		As for Remark \ref{remdexses}, modularity appears in Theorem \ref{hsdses} partly because the diagram \eqref{nmonoseq} consists of three normal subobjects $X$, $Y$, $Z'$ of the object $Z$ where $X$ is smaller than $Y$.
	\end{remark}
	
	The following proves that modularity holds in higher dimensions, as for distributivity in \ref{dexsesn}.
	
	\begin{proposition}\label{hsdsesn}
		Let $\X$ be a z-exact category. Then $\nsub(X)$ is modular for every $X\in\X$ if and only if $\nsub(S)$ is modular for every $S\in\ses(\X)$. When this happens, $\nsub(S)$ is modular for every $S\in\ses^n(\X)$.
	\end{proposition}
	\begin{proof}
		As in the proof of \ref{dexsesn}, we need to understand how modularity is expressed in $\ses(\X)$. Let $VW$ be a short exact sequence, and $X'X$, $Y'Y$ and $Z'Z$ be three normal subobjects of $VW$ such that $X'X$ is smaller than $Y'Y$. With the same computation as in the proof of \ref{dexsesn}, $(X'X\union Z'Z)\inter Y'Y$ is the short exact sequence
		$$
			V\inter(X\union Z)\inter Y\normono(X\union Z)\inter Y\norepi Q
		$$
		for some $Q$ and $X'X\union (Z'Z\inter Y'Y)$ is the short exact sequence
		$$
			V\inter(X\union (Z\inter Y))\normono X\union (Z\inter Y)\norepi Q'
		$$
		for some $Q'$. Therefore, modularity holds in $\nsub(W)$ if and only if it holds in $\nsub(VW)$.
	\end{proof}
	
	Proposition \ref{hsdsesn} and Theorem \ref{hsdses} together show that if $\X$ is a regular and homologically self-dual category in which the lattices of normal subobjects are all modular, then $\ses^n(\X)$ is homologically self-dual for every $n$.
	
	\section{Conclusion}
	
	In \cite{PVdL3}, there is the following lemma:
	
	\begin{lemma}\label{monoreg}\cite[Lemma~2.5.12]{PVdL3}
		Let $\X$ be a category which is both regular and z-exact. For any totally normal sequence of monomorphisms $X\normono Y\normono Z$, if the map $Y/X\rightarrow Z/X$ is a monomorphism, then it is a normal monomorphism.\noproof
	\end{lemma}
	
	In our example of a z-exact category which is not HSD (which is $\ses(\cmon)$, see \ref{sescmonhsd}), this precise map denoted by $\gamma$ is not a normal monomorphism but is nonetheless monic. Thus, Lemma \ref{monoreg} implies that $\ses(\cmon)$ is not regular. More importantly, we could wonder if every regular and z-exact category, and in particular every pointed variety, is homologically self-dual. Actually this is not the case: the category $\mon$ of monoids is a counterexample. This has been proven by Mariano Messora.
	
	Besides, Proposition \ref{dexsesn} points out the fact that di-exactness can be utilized in higher extensions only in some very specific contexts. Namely, when the lattices of normal subobjects are distributive, which excludes abelian categories. The distributivity of the lattice of congruences has already been thoroughly examined (see for example \cite{Pedicchio2,Bourn2001b,Tomasthesis,EGJVdL}), but note that the congruences and the normal subobjects correspond only in ideal determined categories \cite{Janelidze-Marki-Tholen-Ursini}.
	
	Another potential field of investigation is the characterization of modularity and distributivity in the lattice of normal subobjects in terms of how an antinormal map can generate a di-extension. It doesn't look like the converse of Theorem \ref{dexmod} is true, but maybe something weaker than di-exactness could fit.
	
	\section*{Acknowledgments}
	
	Many thanks to Tim Van der Linden for suggesting us to work on this subject, and for all his helpful ideas and advice. We are also grateful to Mariano Messora for his example of a pointed variety which is not homologically self-dual.
	
	
	\providecommand{\noopsort}[1]{}
	\providecommand{\bysame}{\leavevmode\hbox to3em{\hrulefill}\thinspace}
	\providecommand{\MR}{\relax\ifhmode\unskip\space\fi MR }
	\providecommand{\MRhref}[2]{%
		\href{http://www.ams.org/mathscinet-getitem?mr=#1}{#2}
	}
	\providecommand{\href}[2]{#2}

\end{document}